\documentclass[12pt,a4paper,reqno]{amsart}

\usepackage[utf8]{inputenc}
\usepackage[T1]{fontenc}
\usepackage[english]{babel}
\usepackage{amsmath}
\usepackage{amsfonts}
\usepackage{amssymb}
\usepackage{amsthm}
\usepackage{ucs}
\usepackage{graphicx}
\usepackage{xcolor}
\usepackage{dsfont}
\usepackage{wasysym}
\usepackage{physics}
\usepackage{mathtools}
\usepackage{xifthen}
\usepackage{todonotes}
\usepackage{enumitem}
\usepackage{stmaryrd}
\usepackage{constants}
\usepackage[ruled]{algorithm2e}
\usepackage{tikz,pgfplots}
\usepackage{chngcntr}

\newconstantfamily{expoThm}{
	symbol=\nu
}
\newconstantfamily{csts}{
	symbol=C
}
\renewconstantfamily{normal}{
	symbol=c
}

\makeatletter
\newcommand{\pushright}[1]{\ifmeasuring@#1\else\omit\hfill$\displaystyle#1$\fi\ignorespaces}
\newcommand{\pushleft}[1]{\ifmeasuring@#1\else\omit$\displaystyle#1$\hfill\fi\ignorespaces}
\makeatother

\newcommand{\Z}{\mathbb{Z}}
\newcommand{\N}{\mathbb{N}}
\newcommand{\R}{\mathbb{R}}
\newcommand{\bbL}{\mathbb{L}}
\newcommand{\torusd}{\mathbb{T}^d}
\newcommand{\Zd}{\mathbb{Z}^d}

\newcommand{\betac}{\beta_{\mathrm{\scriptscriptstyle c}}}

\newcommand{\normI}[1]{\left\|#1\right\|_{\scriptscriptstyle 1}}

\newcommand{\given}{\,|\,}

\newcommand{\interior}{\mathrm{\scriptscriptstyle int}}
\newcommand{\exterior}{\mathrm{\scriptscriptstyle ext}}
\newcommand{\edge}{\mathrm{\scriptscriptstyle edge}}

\newcommand{\lrangle}[1]{\langle #1\rangle}

\newcommand{\pressure}{\psi}

\newcommand{\e}{\mathbb{E}}
\newcommand{\p}{\mathbb{P}}

\newcommand{\clusterSet}{\mathrm{cl}}

\renewcommand{\nleftrightarrow}{\mathrel{\ooalign{$\leftrightarrow$\cr\hidewidth$/$\hidewidth}}}
\newcommand{\nlongleftrightarrow}{\mathrel{\ooalign{$\longleftrightarrow$\cr\hidewidth$/$\hidewidth}}}

\newcommand{\Ursell}{U}
\newcommand{\weight}{\textnormal{w}}
\newcommand{\partition}{\textnormal{Part}}

\newcommand{\calF}{\mathcal{F}}

\newcommand{\calP}{\mathcal{P}}
\newcommand{\bfP}{\mathbf{P}}

\newcommand{\support}{\textnormal{SUP}}
\newcommand{\update}{\textnormal{UPD}}
\newcommand{\killedUpdate}{\textnormal{KUPD}}
\newcommand{\killedUpdateAbrv}{\textnormal{K}}

\theoremstyle{plain}
\newtheorem{theorem}{Theorem}[section]
\newtheorem{lemma}[theorem]{Lemma}
\newtheorem{claim}[theorem]{Claim}

\newtheorem{corollary}[theorem]{Corollary}

\theoremstyle{definition}
\newtheorem{definition}{Definition}
\newtheorem*{definition*}{Definition}
\newtheorem{remark}{Remark}

\author{S\'{e}bastien Ott}
\address{Section de Mathématiques, Université de Genève, CH-1211 Genève, Switzerland}
\email{sebastien.ott@unige.ch}

\title[Analyticity of the Pressure in the Ising Model]{Weak Mixing and Analyticity of the Pressure in the Ising Model}

\begin{document}
	
\begin{abstract}
	We prove that the pressure (or free energy) of the finite range ferromagnetic Ising model on \(\Zd\) is analytic as a function of both the inverse temperature \(\beta\) and the magnetic field \(h\) whenever the model has the exponential weak mixing property. We also prove the exponential weak mixing property whenever \(h\neq 0\). Together with known results on the regime \(h=0,\beta<\betac\), this implies both analyticity and weak mixing in all the domain of \((\beta,h)\) outside of the transition line \([\betac,\infty)\times \{0\}\). The proof of analyticity uses a graphical representation of the Glauber dynamic due to Schonmann and cluster expansion. The proof of weak mixing uses the random cluster representation.
\end{abstract}

\maketitle

\section{Introduction}

\subsection{Ising Model and the Pressure}

We work with the Ising model on \(\Zd\) with nearest neighbour interactions. The results extend to finite range interactions, but, for the sake of presentation, we will restrict to nearest neighbours. We will work in the uniqueness regime, we thus consider the measure as obtained by limit of finite volume measures on the torus:
\begin{gather*}
	\mu_{N,\beta,h}(\sigma) = \frac{1}{Z_{N,\beta,h}} \exp{ \beta\sum_{i\sim j} \sigma_i\sigma_j + h \sum_{i} \sigma_i}\\
	Z_{N,\beta,h} = \sum_{\sigma\in \{-1,+1\}^{\torusd_N}} \exp{ \beta\sum_{i\sim j} \sigma_i\sigma_j + h \sum_{i} \sigma_i},
\end{gather*}where the sum over \(i\sim j\) is over nearest neighbour pairs in \(\torusd_N\) and the sum over \(i\) is over \(i\in\torusd_N\). \(\beta,h,N\) will often be dropped from the notation. The canonical thermodynamic quantity associated to this model is the \emph{pressure}
\begin{equation}
	\label{eq:pressure}
	\pressure(\beta,h) = \lim_{N\to\infty} \frac{1}{N^d} \log(Z_{N,\beta,h}),
\end{equation}whose non-analyticity points determine the phase transition points of the model. We refer to~\cite{Friedli+Velenik-2017} for proof of this last statement and of \(\pressure\)'s existence. It is directly related to the free energy \(f(\beta,h)=-\frac{1}{\beta}\pressure(\beta,h)\).

\subsection{The Problem of \(\pressure\)'s Analyticity}

A classical way to characterize phase transition is through analytic properties of the pressure of a model. A standard definition of a transition point is thus
\begin{definition}
	\label{def:transitionPoint}
	\((\beta,h)\) is a transition point if \(\pressure\) is not analytic at \((\beta,h)\).
\end{definition}
The phase transition in the Ising model is one of the most studied phenomena in classical equilibrium statistical physic and the phase diagram is known to be represented by Figures~\ref{fig:PhaseDia}.
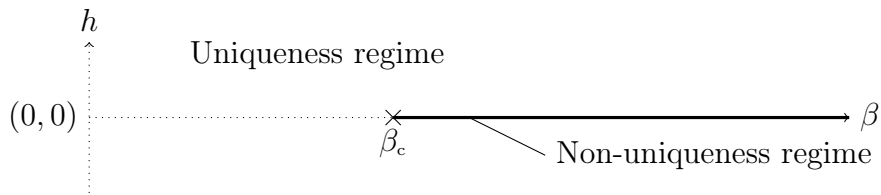
\begin{figure}[h]
	\begin{center}
		\begin{tikzpicture}
		\draw[->,dotted] (0,0) -- (10,0);
		\draw (10,0) node[right] {$\beta$};
		\draw (3.9,0.1)--(4.1,-0.1);
		\draw (3.9,-0.1)--(4.1,0.1);
		\draw (4,0) node[below] {$\betac$};
		\draw(3,0.5)node[above] {Uniqueness regime};
		\draw[very thick] (4,0)--(10,0);
		\draw (5,0)--(6,-0.5);
		\draw(6,-0.5)node[right] {Non-uniqueness regime};
		\draw (0,0) node[left] {$(0,0)$};
		\draw [->,dotted] (0,-1) -- (0,1);
		\draw (0,1) node[above] {$h$};
		\end{tikzpicture}
		\caption{Phase diagram of the Ising model in dimension \(d\geq 2\).}
		\label{fig:PhaseDia}
	\end{center}
\end{figure}

with the regime where at least two phases coexist being \((\betac(d),\infty)\times \{0\}\). The set of transition points should thus be \([\betac(d),\infty)\times \{0\}\). Many properties of the measures are known outside of the transition line: uniqueness of the infinite volume measure, exponential decay of covariances, CLT type result for the block magnetisation field... On the side of characterization of the transition points via Definition~\ref{def:transitionPoint}, one, of course, counts the perturbative results (that can be obtained using cluster expansion for the high and low temperature expansions, but the cited results are anterior to systematic use of this tool).
\begin{itemize}
	\item \(\beta\leq \beta_-<\betac\), \(\pressure\) analytic in \(\beta\) and \(h\) (see for example \cite{Gallavotti+Miracle-Sole+Robinson-1968,Messager+Trotin-1974,Messager+Trotin-1976}).
	\item \(h=0,\beta\geq  \beta_+>\betac\), \(\pressure\) analytic in \(\beta\) (see for example \cite{Messager+Trotin-1974,Messager+Trotin-1976}).
\end{itemize}
As well as a wealth of non-perturbative results
\begin{itemize}
	\item the Lee-Yang Theorem (\cite{Lee+Yang-1952,Yang+Lee-1952}) yields the analyticity of \(\pressure\) in \(h\) on \(\{\real(h)>0\}\).
	\item In dimension \(2\), in the case of nearest-neighbour interactions, Onsager explicitly computed \(\pressure\) in \cite{Onsager-1944}. In more general planar cases, computations of \(\pressure\), based on algebraic or combinatorial methods, are available. See for example~\cite{Kasteleyn-1963,McCoy+Wu-1973,Cimasoni+Duminil-Copin-2013} and references therein.
	\item Alternatively, still for \(d=2\), it is known that weak mixing implies a stronger form of mixing (called strong mixing), and that the later implies complete analyticity of the model (of which the analyticity of the pressure is a trivial consequence). See~\cite{Martinelli+Olivieri+Schonmann-1994,Schonmann+Shlosman-1995}.
	\item \cite{Lebowitz-1972} proves smoothness of \(\pressure\) in both \(h\) and \(\beta\) whenever covariances decay exponentially; together with \cite{Aizenman+Barsky+Fernandez-1987} and \cite{Lebowitz+Penrose-1968}, this yields smoothness of \(\pressure\) outside of the half line \([\betac,\infty)\times\{0\}\).
	\item \cite{Martin-Lof-1973} where smoothness of \(\beta\mapsto\pressure(\beta,0)\) is proved in the regime \(\beta\geq\beta_+>\betac\), the proof works under the assumption that the covariances decay exponentially with the distance in a pure state. Toghether with \cite{Duminil-Copin+Goswami+Raoufi-2018}, this implies smoothness in the regime \(\beta>\betac\). It is also shown, under the same hypotheses, that \(\pressure\) possesses directional derivatives at all orders in \(h\) at \(h=0\).
\end{itemize}
All together, this results give smoothness whenever \(\pressure\) is expected to be smooth (see Figure~\ref{fig:SmoothnessDia}) and analyticity in the regimes depicted in Figure~\ref{fig:AnalyticityDia}.
\begin{figure}[h]
	\begin{center}
		\begin{tikzpicture}
		\shade[bottom color=lightgray,top color=white] (0,0.5) -- (0,1) -- (10,1) -- (10,0.5) -- cycle;
		\fill[lightgray](0,-0.5) -- (0,0.5) -- (10,0.5) -- (10,-0.5) -- cycle;
		\shade[bottom color=white,top color=lightgray] (0,-1) -- (0,-0.5) -- (10,-0.5) -- (10,-1) -- cycle;
		\draw[->,dotted] (0,0) -- (10,0);
		\draw (10,0) node[right] {$\beta$};
		\draw (3.9,0.1)--(4.1,-0.1);
		\draw (3.9,-0.1)--(4.1,0.1);
		\draw (4,0) node[below] {$\betac$};
		\draw[very thick] (4,0)--(10,0);
		\draw (0,0) node[left] {$(0,0)$};
		\draw [->,dotted] (0,-1) -- (0,1);
		\draw (0,1) node[above] {$h$};
		\end{tikzpicture}
		\caption{Domain where smoothness of \(\pressure\) is proven. The black line is the region where \(\pressure\) is smooth only as a function of \(\beta\).}
		\label{fig:SmoothnessDia}
	\end{center}
\end{figure}
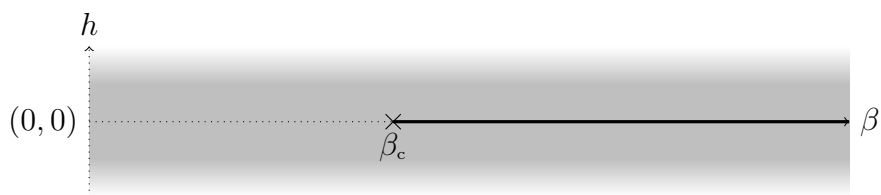

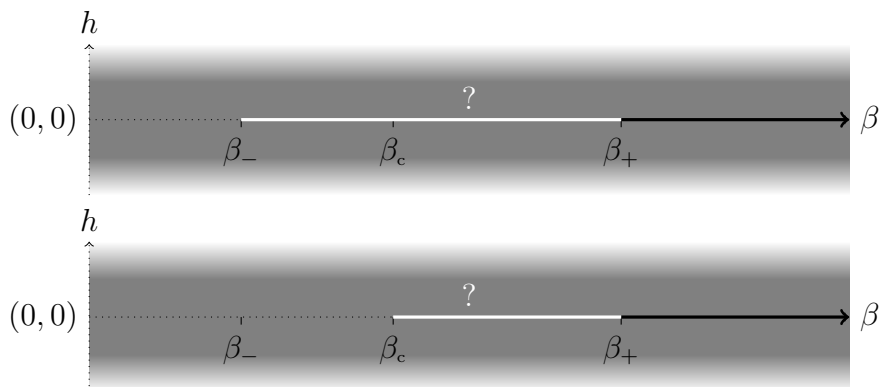
\begin{figure}[h]
	\begin{center}
		\begin{tikzpicture}
		\shade[bottom color=gray,top color=white] (0,0.5) -- (0,1) -- (10,1) -- (10,0.5) -- cycle;
		\fill[gray](0,-0.5) -- (0,0.5) -- (10,0.5) -- (10,-0.5) -- cycle;
		\shade[bottom color=white,top color=gray] (0,-1) -- (0,-0.5) -- (10,-0.5) -- (10,-1) -- cycle;
		\draw[->,dotted] (0,0) -- (10,0);
		\draw (10,0) node[right] {$\beta$};
		\draw (4,0)--(4,-0.1) node[below] {$\betac$};
		\draw[->,very thick] (7,0)--(10,0);
		\draw (2,0)--(2,-0.1) node[below] {$\beta_-$};
		\draw (7,0)--(7,-0.1) node[below] {$\beta_+$};
		\draw[very thick, white] (2,0)--(7,0);
		\draw[white] (5,0) node[above] {?};
		\draw (0,0) node[left] {$(0,0)$};
		\draw [->,dotted] (0,-1) -- (0,1);
		\draw (0,1) node[above] {$h$};
		\end{tikzpicture}
		\begin{tikzpicture}
		\shade[bottom color=gray,top color=white] (0,0.5) -- (0,1) -- (10,1) -- (10,0.5) -- cycle;
		\fill[gray](0,-0.5) -- (0,0.5) -- (10,0.5) -- (10,-0.5) -- cycle;
		\shade[bottom color=white,top color=gray] (0,-1) -- (0,-0.5) -- (10,-0.5) -- (10,-1) -- cycle;
		\draw[->,dotted] (0,0) -- (10,0);
		\draw (10,0) node[right] {$\beta$};
		\draw (4,0)--(4,-0.1) node[below] {$\betac$};
		\draw[->,very thick] (7,0)--(10,0);
		\draw (2,0)--(2,-0.1) node[below] {$\beta_-$};
		\draw (7,0)--(7,-0.1) node[below] {$\beta_+$};
		\draw[very thick, white] (4,0)--(7,0);
		\draw[white] (5,0) node[above] {?};
		\draw (0,0) node[left] {$(0,0)$};
		\draw [->,dotted] (0,-1) -- (0,1);
		\draw (0,1) node[above] {$h$};
		\end{tikzpicture}
		\caption{Domain where the analyticity of \(\pressure\) is known. The black line is the region where analyticity only holds for \(\beta\). The top picture is the state of the problem before this work. This article covers the grey part of the bottom picture.}
		\label{fig:AnalyticityDia}
	\end{center}
\end{figure}

The goal of this article is to close the High Temperature side of the problem by showing that \(\pressure\) is analytic in both \(h\) and \(\beta\) at any point \((\beta,0)\) with \(\beta<\betac\). The proof implies analyticity of \(\pressure\) in \(\beta\) and \(h\) around every point not in \([\betac,\infty)\times \{0\}\). This also close the problem of analyticity in \(h\) and leaves open the problem of proving that \(\beta\mapsto\pressure(\beta,0)\) is analytic in \(\beta\) on \((\betac,\infty)\).

\begin{remark}
	Analyticity of the pressure (as well as convergent (uniformly over volumes) cluster expansions for partition functions) are a consequence of the \emph{complete analyticity} conditions of~\cite{Dobrushin+Shlosman-1987}, or of some of its restricted version (see the discussion in~\cite{Martinelli+Olivieri-1994}). Equivalence with results on the mixing rate (uniformly over volumes and boundary conditions) of Glauber dynamic is investigated in~\cite{Stroock+Zegarlinski-1992,Yoshida-1997}. The main difference is that the result needed in the present work is exponential relaxation of the \emph{infinite volume} dynamic (which holds throughout the whole off-transition region via weak mixing) and not of the dynamic in any finite volume with any boundary conditions (which fails at low temperature and small positive field in dimension \(\geq 3\), see~\cite{Basuev-2007}).
\end{remark}

\subsection{Notations, Conventions and a Few Definitions}
\label{sec:MetricGraphProp}
We write \([n]=\{1,2,...,n\}\). For \(G=(V_G,E_G)\) a graph we denote \(i\sim j \iff \{i,j\}\in E_G\). For \(C\subset V_G\), we denote:
\begin{gather*}
	\partial^{\interior}C = \big\{i\in C:\ \exists j\in V_G\setminus C, \{i,j\}\in E_G\big\},\\
	\partial^{\exterior}C = \big\{i\in V_G\setminus C:\ \exists j\in C, \{i,j\}\in E_G\big\} ,\\
	\partial^{\edge}C=\big\{\{i,j\}\in E_G:\ i\in C, j\in V_G\setminus C \big\}.
\end{gather*}
For a set \(V\) we denote \(\calP(V)=\{A\subset V\}\) the set of subsets of \(V\) and \(\partition(V)\) the set of (unordered) partitions of \(V\). We see \(\Zd\) and \(\torusd_N\) as canonically embedded in \(\R^d\), \((-N,N]^d\) respectively. Denote
\begin{equation*}
	B_l=\{v \in \Zd:\ \norm{v}_{\infty} \leq l \}, \quad B_l(v) = B_l+ v.
\end{equation*}

Also define the sub-lattices \(\bbL_L = ((2L+1)\Z)^d\). \(\Zd\) is then naturally paved by the set of (disjoint) boxes \((B_L(v))_{v\in\bbL_L}\). We will also use \(\bbL_L\) for the subset of \(\torusd_N\) without mention when clear from the context (and will then make the implicit assumption that \(2N\) is divisible by \(2L+1\)). We will also use the following notation: for a set \(\Delta\subset \Zd\) let
\begin{equation*}
	[\Delta]_L = \{v\in \bbL_L:\ \Delta\cap B_L(v)\neq \varnothing \}.
\end{equation*}
We will also often see \([\Delta]_L\) as a subset of \(\Zd\) via \(\bigcup_{v\in [\Delta]_L} B_L(v)\). When doing so, we add ``seen as a subset of \(\Zd\)'' after \([\Delta]_L\). \([\Delta]_L\) is a ``coarse approximation'' of \(\Delta\). Notice that \([\Delta]_L\) (as subset of \(\Zd\)) is connected if \(\Delta\) is.

We say that a measure \(\mu\) on \(S^{V_G}\) is \emph{weak mixing} if there exists \(C\geq 0,c>0\) such that for any \(\Delta,\Delta'\subset V_G\) and any events \(A,B\) supported on \(\Delta, \Delta'\) respectively with \(\mu(B)>0\), one has
\begin{equation*}
	|\mu(A\given B)- \mu(A)| \leq C\sum_{i\in \Delta,j\in\Delta'} e^{-cd_G(i,j)},
\end{equation*}where \(d_G\) is the graph distance in \(G\). We denote this property \(WM(\mu)\).

We say that a measure \(\mu\) on \(S^{V_G}\) is \emph{ratio weak mixing} if there exists \(C\geq 0,c>0\) such that for any \(\Delta,\Delta'\subset V_G\) and any event \(A,B\) having strictly positive probability, supported on \(\Delta, \Delta'\) respectively, one has
\begin{equation*}
\Big|\frac{\mu(A\cap B)}{\mu(A)\mu(B)}- 1\Big| \leq C\sum_{i\in \Delta,j\in\Delta'} e^{-cd_G(i,j)},
\end{equation*}where \(d_G\) is the graph distance in \(G\).

\subsection{Results}

\begin{theorem}
	\label{thm:Ising_weakMixing}
	If \(h>0\), then \(\mu_{\beta,h}\) is weak mixing.
\end{theorem}
\begin{proof}
	The weak mixing property is known to be a consequence of the exponential relaxation of the magnetization (see \cite[Section 4.1]{Martinelli-1999}). Thus, Theorem~\ref{thm:epxRelax_Magnet} implies Theorem~\ref{thm:Ising_weakMixing}.
\end{proof}

\begin{corollary}
	If \(h>0\), then \(\mu_{\beta,h}\) is ratio weak mixing.
\end{corollary}
\begin{proof}
	\ref{thm:Ising_weakMixing} and the spatial Markov property imply that the hypotheses of~\cite[Theorem 3.3]{Alexander-1998} are fulfilled. Thus, \(\mu_{\beta,h}\) is ratio weak mixing.
\end{proof}

\begin{theorem}
	\label{thm:PressureAnalytic}
	For any \(d\geq 1\) and points \((\beta,h)\) such that \(\mu_{\beta,h}\) is weak mixing, \(\pressure\) is analytic in a neighbourhood of \((\beta,h)\).
\end{theorem}

\begin{corollary}
	\label{cor:PressureAnalytic}
	For any \(d\geq 1\) and points \((\beta,h)\) with \(h\neq 0\) or \(h=0\) and \(\beta<\betac(d)\), \(\pressure\) is analytic in a neighbourhood of \((\beta,h)\).
\end{corollary}
\begin{proof}
	By Theorem~\ref{thm:Ising_weakMixing} \(\mu_{\beta,h}\) is weak mixing whenever \(h\neq 0\). By~\cite{Aizenman+Barsky+Fernandez-1987}, \(\mu_{\beta,0}\) is weak mixing whenever \(\beta<\betac\) (see also~\cite{Duminil-Copin+Tassion-2016} for a more recent proof). Apply then Theorem~\ref{thm:PressureAnalytic}.
\end{proof}

\begin{theorem}
	\label{thm:CorrFunctAnalytic}
	For any \(d\geq 1\) and points \((\beta,h)\) such that \(\mu_{\beta,h}\) is weak mixing, for any \(A\subset \Zd\) finite, \((z, w)\mapsto\lrangle{\sigma_A}_{\beta+z,h+w}\) is analytic in a neighbourhood of \((0,0)\).
\end{theorem}

\begin{remark}
	We state here the results for nearest-neighbours interactions but the proof works the same for finite range interactions at the cost of heavier notations. The proof also implies that the pressure is analytic in any small enough perturbation of the Hamiltonian by a finite range potential.
\end{remark}

\begin{remark}
	As well as Theorems~\ref{thm:PressureAnalytic} and~\ref{thm:CorrFunctAnalytic}, the proof provides a way to construct a convergent cluster expansions for perturbations of the partition functions on the torus. Moreover, working with the Ising model is only required to have Theorem~\ref{thm:ExpDecSUP}, whose proof uses the lattice FKG property. But this result should hold for most lattice spin models with the weak mixing property.
\end{remark}

\subsection{Organization of the Paper}

The proof of Theorem~\ref{thm:Ising_weakMixing} (i.e.: the proof of exponential relaxation of the magnetisation) is independent of the rest and is contained in Section~\ref{sec:ExpRelaxMagnet}. Section~\ref{sec:DepEncMeas_Polymer} defines the objects that will be used later on, Section~\ref{sec:InfoPerco} contains a short presentation of the graphical representation associated to Glauber dynamic and of what is information percolation. Section~\ref{sec:CoarseGraining} contains the key estimates needed in the proof of analyticity. Finally, Section~\ref{sec:AnalyticityProof} wraps up things together and conclude the proof of Theorems~\ref{thm:PressureAnalytic} and~\ref{thm:CorrFunctAnalytic}.

\section{Dependency Encoding Measures and Associated Polymer Measures}
\label{sec:DepEncMeas_Polymer}

We start by introducing the key notion in the analysis that will follow. Let \(\Lambda\) and \(\Omega\) be two sets: the space and the spin values. Let \(\mu\) be a probability measure on \(\Omega^{\Lambda}\).
\begin{definition}
	A measure \(\Phi\) on \(\Omega^{\Lambda}\times \{A\subset \Lambda\}^{\Lambda}\) is said to be \emph{dependency encoding} for \(\mu\) if, for \((\sigma,X)\sim \Phi\),
	\begin{enumerate}
		\item The first marginal of \(\Phi\) is \(\mu\), i.e. \(\sigma\sim \mu\).
		\item \(v\in X_v\).
		\item If \(f,g:\Omega^{\Lambda}\to\R\) are measurable functions supported on \(\Delta,\Delta'\subset\Lambda\) and \(C,C'\subset\Lambda\) are disjoint sets,
		\begin{equation*}
		\Phi\big(f(\sigma)g(\sigma) \mathds{1}_{X_{\Delta} = C}\mathds{1}_{X_{\Delta'} = C'}\big) = \Phi\big(f(\sigma) \mathds{1}_{X_{\Delta} = C}\big)\Phi\big(g(\sigma) \mathds{1}_{X_{\Delta'} = C'}\big),
		\end{equation*}
		where \(X_{\Delta} = \bigcup_{v\in \Delta} X_v\).
	\end{enumerate}
\end{definition}

For functions \(f_1,...,f_n:\Omega^{\Lambda}\to\R\) supported on \(\Delta_1,...,\Delta_n\) we can decompose \(\Phi(f_1...f_n)\) according to the ``maximal connected components'' of \(\bigcup_{i=1}^{n} X_{\Delta_i}\):
\begin{align*}
\Phi(f_1...f_n) &= \sum_{C_1,...,C_n\subset\Lambda} \Phi\big(f_1...f_n\mathds{1}_{X_{\Delta_1}=C_1}...\mathds{1}_{X_{\Delta_n}=C_n}\big)\\
&= \sum_{B\in\partition([n])} \sum_{(C_b)_{b\in B}}\prod_{\{b,b'\}\subset B} \mathds{1}_{C_b\cap C_{b'}=\varnothing} \prod_{b\in B} \Phi\big(f_b\mathds{1}_{A(C_b,(\Delta_i)_{i\in b})}\big)
\end{align*}
where the second summation is over collections of subsets of \(\Lambda\) indexed by \(B\) and we introduced \(f_b=\prod_{i\in b} f_i\) and
\begin{equation*}
A\big(C,(\Delta_i)_{i\in I}\big) = \big\{ \bigcup_{i\in I} X_{\Delta_i} = C\big\} \cap D\big((\Delta_i)_{i\in I}\big)^c,
\end{equation*}where
\begin{multline*}
D\big((\Delta_i)_{i\in I}\big) = \Big\{ \exists I_1,I_2\neq\varnothing:\ I_1\cap I_2=\varnothing,\ I_1\cup I_2= I,\\ \bigcup_{i\in I_1}X_{\Delta_i} \cap \bigcup_{i\in I_2}X_{\Delta_i} = \varnothing \Big\}.
\end{multline*}
Summing over ordered partitions instead of partitions gives
\begin{multline}
\label{eq:PolymerRep_ProdObs}
\Phi(f_1...f_n) = \sum_{m=1}^{n}\frac{1}{m!}\sum_{B_1\sqcup...\sqcup B_m= [n]}\sum_{C_1,...,C_m}\\ \prod_{\{i,j\}\subset [m]} \mathds{1}_{C_i\cap C_{j}=\varnothing} \prod_{i=1}^m \Phi\big(f_{B_i}\mathds{1}_{A(C_{i},(\Delta_j)_{j\in B_i})}\big),
\end{multline}where the sum over \(B_1,..., B_m\) is over non-empty disjoint subsets of \([n]\).

\subsection*{Polymer Model Associated with a Product of Functions} Take \(\Phi\) dependency encoding for \(\mu\). Take a family of functions \((\tilde{f}_A)_{A\subset \Lambda}\), \(\tilde{f}_A:\Omega^{A}\to\mathbb{C}\). Defining \(f_A=\tilde{f}_A-1\), we have
\begin{equation*}
	\lrangle{\prod_{A\subset \Lambda} \tilde{f}_A} = \sum_{H\subset \calP(\Lambda)} \lrangle{\prod_{A\in H} f_A}.
\end{equation*}
For \(C\subset \Lambda\) define
\begin{equation*}
	\weight(C) = \sum_{H\subset\calP(C)} \Phi\Big( \prod_{A\in H}f_A \mathds{1}_{A( C,H )} \Big).
\end{equation*}
Using \eqref{eq:PolymerRep_ProdObs}, we obtain
\begin{equation}
\label{eq:PolymerRep_General}
	\lrangle{\prod_{A\subset \Lambda} \tilde{f}_A} = 1+\sum_{m\geq 1} \frac{1}{m!} \sum_{C_1,...,C_m} \prod_{\{i,j\}\subset[m]}\delta(C_i,C_j) \prod_{i=1}^m \weight(C_i),
\end{equation}
with \(\delta(C_i,C_j) = \mathds{1}_{C_i\cap C_{j}=\varnothing} \). This is the partition function of a polymer model as described in Appendix~\ref{app:ClusterExp} with polymers being subsets of \(\Lambda\).

\section{Information Percolation}
\label{sec:InfoPerco}

The goal of this Section is to construct a dependency encoding measure for \(\mu_{\beta,h}\). Properties of this measure will then be studied in Section~\ref{sec:CoarseGraining}.

Information percolation is a way to encode the dependencies between regions of space and time for a configuration sampled using a Glauber dynamic. The graphical representation of information dates back to the work of Schonmann~\cite{Schonmann-1994} and was then exploited by Martinelli and Olivieri in~\cite{Martinelli+Olivieri-1994,Martinelli-1999}. It is an instrumental tool in Lubetzky and Sly's proof, \cite{Lubetzky+Sly-2015}, of cut-off for the mixing time of the Glauber dynamic associated to the Ising model. 

\subsection*{Glauber Dynamic}
We consider the Ising model on \(\torusd_N\) (treating \(N=\infty\) via \(\torusd_{\infty}=\Zd\)) at inverse temperature \(\beta<\betac\) with no magnetic field, denoted \(\mu_N\) (the \(N\) will often be dropped from the notation). A classical way to sample a configuration of this model is the Glauber dynamic: consider a continuous time Markov chain on \(\{-1,+1\}^{\torusd_N}\) with generator
\begin{equation*}
	(Lf)(\sigma) =\sum_{v\in \torusd_N} c(v,\sigma)(f(\sigma^v)-f(\sigma)),
\end{equation*} where \(c(v,\sigma)\) are the flipping rates and \(\sigma^v\) denote the configuration obtained from \(\sigma\) by flipping the spin at \(v\). A classical instance of this dynamic is the \emph{heat bath} dynamic where \(c(v,\sigma) = \big(1+e^{-2\beta\sigma(v)\sum_{u\sim v} \sigma(u)}\big)^{-1} \); a graphical interpretation of this dynamic is the following: equip each site with a rate \(1\) Poisson clock and when one clock rings, re-sample the associated site \(v\) according to the measure \(\mu\) conditioned on the value of \(\sigma(u)\ \forall u\neq v \).

One can construct this dynamic as follows:
\begin{itemize}
	\item To each site \(v\in\torusd\), attach an independent copy of \(\big(T,(U_i)_{i\in \N}\big)\) where \(T\) is a Poisson point process of intensity \(1\) on \(\R_{\leq 0}\) and \((U_i)_i\) is an i.i.d. family of uniform random variable on \([0,1]\) (also independent of \(T\)). Denote \(\p,\e\) the law and expectation of this whole family. For a set of sites \(\Delta\), denote \(\calF_{\Delta}\) the sigma algebra generated by the Poisson point processes and the uniform random variables attached to a site in \(\Delta\).
	\item For a time interval \([-t,0]\) we have that the probability of at least two clocks ringing at the same time is \(0\), we can thus look at the totally ordered \emph{update time sequence} \(-t \leq -t_m < ... < -t_1 \leq 0\) (the superposition of the Poisson point processes) and order accordingly the \emph{updated sites sequence} \(v_m,...,v_1\) (\((t_i)_i\) is the sequence of times at which one of the clocks ringed and \(v_i\) is the site at which is attached the clock that ringed at time \(t_i\)) and the \emph{flip probability sequence} \(0\leq u_m,...,u_1 \leq 1\) such that \(u_i\) is the value of \(U^{v_i}_k\) the \(k^{\text{th}}\) uniform random variable attached to \(v_i\), where \(k\) is the number of times \(v_i\) appears in the sequence \(v_{i-1},...,v_1\).
	\item For a given starting configuration \(\eta\in\{-1,+1\}^{\torusd}\), a time interval \([-t,0]\) and an update sequence \((-t_i,v_i,u_i)_{i=1}^{m}\), define the process \((\sigma_s)_{s\in[-t,0]}\) by setting \(\sigma_{-t}=\eta\) and by updating \(\sigma\) at each update time \(t_i\) by: 1) keeping it constant on all \(u\neq v_i\), 2) setting:
	\begin{equation}
		\sigma_{t_i}(v_i) = \begin{cases}
			+1 & \text{if } u_i < \big(1+ e^{-2\beta A }\big)^{-1},\\
			-1 & \text{else},
		\end{cases}
	\end{equation}where \(A\) is the sum of the spins neighbouring \(v_i\). \(\sigma_{s}\) is thus constant on the intervals \([-t_i,-t_{i-1})\). \(\sigma_s\) is a deterministic function of \(\eta\) and the update sequence. Moreover, the lattice FKG property of the model implies that, for a given update sequence, it is a non-decreasing function of \(\eta\).
\end{itemize}

For \(0\geq -t\), we will denote \(\big(\sigma_s^{\eta,t}\big)_{s\in[-t,0]}\) the process on the interval \([-t,0]\) with starting configuration \(\eta\). As \(\mu\) is an invariant measure for this chain and the system is finite, we have
\begin{equation}
	\lrangle{f} = \e[\lim_{t\to\infty} f(\sigma_{0}^{\eta,t})],
\end{equation}
for any \(f\) local and any \(\eta\).

For a given (infinite) update sequence and for a pair of sites \((v,-t),(v',-t')\in\torusd\times \R_{-}\) with \(t>t'\geq 0\), we say that \((v',-t')\) is connected to \((v,-t)\) (denoted \((v',-t')\rightarrow(v,-t)\)) if there exists an update sequence \((-t_i,v_i,u_i)_{i=1}^m\) with
\begin{itemize}
	\item \(v_1= v'\), \(v_{i+1}\sim v_{i},\ i=1,...,m-1\), \(v_m\sim v\);
	\item \(t' \leq t_1\leq t_2 \leq...\leq t_m \leq t\);
	\item The intervals \(\{v_i\}\times[-t_{i-1},t_i)\) for \(i=1,...,m\) are free of updates (where we denoted \(t_0=t'\)).
\end{itemize}

\begin{figure}[h]
	\begin{center}
		\includegraphics[scale=0.6]{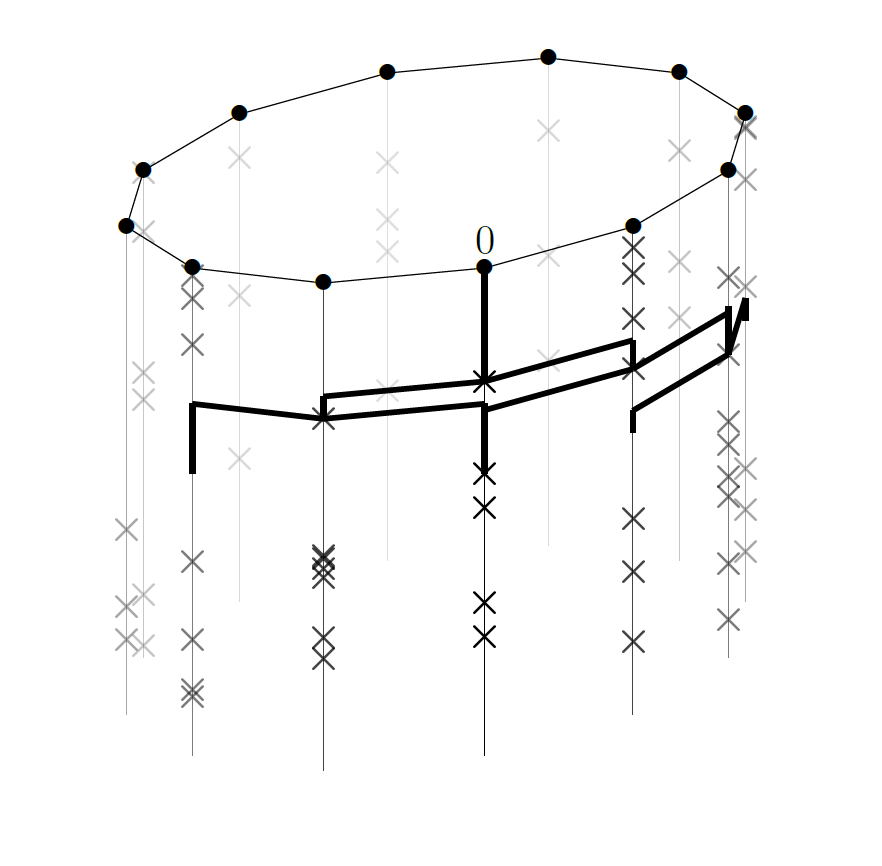}
		\caption{Graphical representation of the Glauber dynamic in dimension \(1\). The marks are the update times. In bold is represented the set of sites connected to \(0\) up to the time of death of the support of \(0\).}
		\label{fig:InformationPerco}
	\end{center}
\end{figure}

\subsection*{Update and Support Functions}
For \(t>t'\geq 0\), \(A\subset\torusd\) and a realization of the update sequence \((t_i,v_i,u_i)_{i=1}^{m}\) during the interval \([-t,-t']\),
\begin{itemize}
	\item Define \(\support(t,t',A)\) to be the \emph{support function} of \(A\) at time \(-t'\) for the dynamic started at time \(-t\): the set of sites \(\Delta\) such that the function \(  \eta\mapsto \sigma_{-t'}^{t,\eta}|_A\) (the configuration \(\sigma_{-t'}^{t,\eta}\) restricted to \(A\)) has support \(\Delta\). In other words, it is the minimal information about time \(-t\) that one needs to reconstruct the configuration at time \(-t'\) on \(A\).
	\item For a set \(A\), define \(\update(t,t',A) = \{v\in\torusd:\ \exists w\in A, s\in [-t,-t'],\ (w,-t')\rightarrow (v,s) \}\), the set of sites that \(A\) reaches through the update sequence in the time interval \([-t,-t']\).
\end{itemize}

Both \(\support\) and \(\update\) are random functions under \(\p\). We refer to \cite{Lubetzky+Sly-2015} for illustrations and a more thorough discussion of those functions properties. To get an intuition on what is going on, one can remark that when the uniform random variable governing an update takes values sufficiently close to \(1\) or \(0\), the update is done uniformly over the values of the neighbouring sites (which participates to the decrease of the support function). Notice that for a given realization of the update sequence \((-t_i,v_i,u_i)_{i=1}^{\infty}\) on the time interval \((-\infty,0]\), both \(t\mapsto \update(t,t',A)\) and \(t\mapsto\mathds{1}_{\support(t,t',A)=\varnothing}\) are non-decreasing functions of \(t\geq t'\). For a set \(A\) and a time \(t'\geq 0\), define the \emph{coupling time} of \(A\times \{-t'\}\) by:
\begin{equation}
\label{eq:CouplingTimeDef}
	\tau_A(t') = \inf\{t\geq t':\ \support(t,t',A)=\varnothing\}.
\end{equation}
Define
\begin{equation*}
	\sigma_{t'}|_A = \lim_{\epsilon\to 0} \sigma_{t'}^{\eta,\tau_A(t')+\epsilon}|_A.
\end{equation*} This limit does not depend on \(\eta\) by definition of \(\tau_A(t')\). Notice that \(\sigma_{t'}^{\eta,t}|_A=\sigma_{t'}|_A\) for any \(\eta\) and any \(t>\tau_A(t')\). In particular,
\begin{equation*}
	\lim_{t\to\infty} \sigma_{t'}^{\eta,t}|_A = \sigma_{t'}|_A,
\end{equation*}
and this limit does not depend on \(\eta\) as the system is finite. Finally, define
\begin{equation}
	\killedUpdate(A,t') = \update(\tau_A(t'),t',A).
\end{equation}
\(\killedUpdate(A,t')\) has the following properties (the second one is a consequence of the first one and of the construction)
\begin{enumerate}[label=(P\arabic*)]
	\item \label{KUPD_prop:Meas} For any \(A\subset\Delta,t'\geq 0\), the event \(\{\killedUpdate(A,t') = \Delta\}\) is \(\calF_{\Delta}\)-measurable.
	\item \label{KUPD_prop:config_Meas} For any set \(A\), and any \(\Delta\supset A\), \(\sigma_{t'}|_A\) is \(\calF_{\Delta}\)-measurable conditionally on \(\{\killedUpdate(A,t') \subset \Delta\}\).
\end{enumerate}

We will denote
\begin{gather*}
	\support(t,A) \equiv \support(t,0,A),\quad \update(t,A) \equiv \update(t,0,A),\\ \tau_A \equiv \tau_A(0),\quad \killedUpdate(A) \equiv \killedUpdate(A,0).
\end{gather*}

We can now describe our dependency encoding measure. Let \(\Phi\) be the marginal of \(\p\) on \(\sigma^{\infty} = \lim_{t\to\infty} \sigma_{0}^{\eta,t}\) and \(X=(\killedUpdate(v))_{v\in \torusd_N}\). By finite energy and finiteness of \(\torusd_N\), \(\sigma^{\infty}\) is almost surely well defined and does not depend on \(\eta\). By properties~\ref{KUPD_prop:Meas} and~\ref{KUPD_prop:config_Meas}, one has
\begin{multline*}
	\Phi\big( \mathds{1}_{\sigma^{\infty}|_{\Delta} =\eta }\mathds{1}_{\sigma^{\infty}|_{\Delta'} =\eta' } \mathds{1}_{\killedUpdate(\Delta)=C} \mathds{1}_{\killedUpdate(\Delta')=C'} \big) = \\
	=\Phi\big( \mathds{1}_{\sigma^{\infty}|_{\Delta} =\eta } \mathds{1}_{\killedUpdate(\Delta)=C} \big)\Phi\big( \mathds{1}_{\sigma^{\infty}|_{\Delta'} =\eta' } \mathds{1}_{\killedUpdate(\Delta')=C'} \big),
\end{multline*}for any \(\Delta,\Delta'\subset \torusd_N\), \(\eta\in\{-1,+1\}^{\Delta}\), \(\eta'\in\{-1,+1\}^{\Delta'}\) and \(C\supset \Delta, C'\supset \Delta'\) with \(C\cap C'=\varnothing\). Moreover, by definition, \(v\in\killedUpdate(v)\), and, by construction, \(\sigma^{\infty}\sim \mu_{N}\). Thus, \(\Phi\) is dependency encoding for \(\mu_{N}\).

In the same fashion, we can construct dependency encoding measures for \(\mu\) seen as a measure on \(S^{\bbL_L}\) with \(S=\{-1,+1\}^{B_L}\) (simply look at blocks of spin as the new spin) by setting \(X(v)=[\killedUpdate(B_L(v))]_L\) for each \(v\in\bbL_L\). We denote these measures by \(\Phi^L\).

The main input we will need is a result of Martinelli and Olivieri \cite[Section 3]{Martinelli+Olivieri-1994} (see also \cite[Section 4.1]{Martinelli-1999})
\begin{theorem}
	\label{thm:ExpDecSUP}
	For any point \((\beta,h)\) such that \(\mu_{\beta,h}\) is weak mixing, there exist \(\Cl{expDecSUP_rate}>0,\Cl{expDecSUP_cst}> 0\) such that for any \(v\in\torusd\) and any \(t,t'\geq 0\),
	\begin{equation*}
	\p(\support(t+t',t,v)\neq \varnothing) \leq \Cr{expDecSUP_cst}e^{-\Cr{expDecSUP_rate} t'}.
	\end{equation*}
\end{theorem}

This result is the only place where we use that we are working with the Ising model (the lattice FKG property of the Ising model is used in the proof of Theorem~\ref{thm:ExpDecSUP}). The rest of the construction only depends on the spatial Markov property and on finite range interactions (even the finiteness of the spin space is not used, the whole construction thus would go through for the Potts model or for spin \(O(N)\) models if Theorem~\ref{thm:ExpDecSUP} was proved to hold in those cases).

\section{Coarse Graining of \(\killedUpdate\)}
\label{sec:CoarseGraining}

In this whole Section, we make the implicit assumption that Theorem~\ref{thm:ExpDecSUP} holds. Its validity thus has to be added in the hypotheses of each statement. We also work with \(d\geq 1\) fixed, the constants appearing can depend on \(d\). The goal of this Section is the proof of

\begin{lemma}
	\label{lem:ExpDecKilledUpdate}
	There exist \(\Cl[expoThm]{expDecKUPD_L}>0, \Cl[expoThm]{SetVolCostKUPD_L}\geq 0\) and \(L_0\geq 0\) such that for any \(L\geq L_0\), \(V\subset\bbL_L\) and \(N\) large enough (as function of \(V\)) one has,
	\begin{equation}
	\label{eq:ExpDecG(V)}
	\p\big(|[\killedUpdate(\bar{V})]_L| \geq M\big)\leq e^{-\Cr{expDecKUPD_L} M L} e^{\Cr{SetVolCostKUPD_L}|V|L},
	\end{equation}
	where \(\bar{V} = \bigcup_{v\in V} B_L(v) \) and \(N\) appears in an implicit fashion as the size of the system we work in.
\end{lemma}

To lighten notation, for the remainder of this Section we will write \(\killedUpdate\equiv \killedUpdateAbrv\).

A direct consequence of Theorem~\ref{thm:ExpDecSUP} is
\begin{lemma}
	\label{lem:ptToBoxDec}
	There exist \(\Cl{expDecUPD_rate}>0,\Cl{expDecUPD_cst}> 0,\epsilon_0>0\) such that for any \(l\geq 0\) and any \(v\in\torusd\),
	\begin{equation}
	\label{eq:ExpDecUPD}
	\p\big( \{\killedUpdateAbrv(v,t)\not\subset B_{l}(v)\} \cup \{\tau_{v}(t) \geq \epsilon_0 l+t\} \big) \leq \Cr{expDecUPD_cst}e^{-\Cr{expDecUPD_rate}l}.
	\end{equation}
\end{lemma}
\begin{proof}
	It is sufficient to prove the result for \(t=0\). Let \(\epsilon>0\) be small to be chosen later. First,
	\begin{align*}
		&\p\big( \{\killedUpdateAbrv(v)\not\subset B_{l}(v)\} \cup \{\tau_{v} \geq \epsilon l\} \big) \leq\\
		&\leq \p\big( \killedUpdateAbrv(v)\not\subset B_{l}(v), \tau_{v} < \epsilon l \big) + \p\big( \support(\epsilon l, v)\neq \varnothing \big)\\
		&\leq \p\big( \killedUpdateAbrv(v)\not\subset B_{l}(v), \tau_{v} < \epsilon l \big) + \Cr{expDecSUP_cst}e^{-\Cr{expDecSUP_rate}\epsilon l},
	\end{align*}
	by Theorem~\ref{thm:ExpDecSUP}. The remaining probability is then upper bounded by the probability that there exists an update path of length at least \(l\) in the time interval \([-\epsilon l, 0]\) starting at \(v\). For such a given path with \(k\) vertices, the probability that this path occurs is bounded from above by the probability that a Poisson random variable of parameter \(\epsilon l\) is \(\geq k\). Bennett's inequality and a union bound give
	\begin{equation*}
		\p\big( \killedUpdateAbrv(v)\not\subset B_{l}(v), \tau_{v} < \epsilon l \big) \leq \sum_{k=l}^{\infty}(2d)^{k} e^{-k(\log(k/\epsilon l) -1)}\leq e^{-l+1},
	\end{equation*}
	if we choose \(\epsilon \leq e^{-2-\log(2d)}\). Taking \(\epsilon_0= e^{-2-\log(2d)}\), \(\Cr{expDecUPD_cst} = \Cr{expDecSUP_cst}+e\) and \(\Cr{expDecUPD_rate} = \min(\Cr{expDecSUP_rate}\epsilon_0, 1)\), we get the claim.
\end{proof}
In particular, for any \(A\) finite, \(\killedUpdateAbrv(A)\) is a.s. uniformly bounded in \(N\) and thus well defined also for the dynamic in infinite volume. From now on, \(\epsilon>0\) will be fixed and equal to the \(\epsilon_0\) provided by Lemma~\ref{lem:ptToBoxDec} (\(\epsilon= e^{-2-\log(2d)}\)).

Define then
\begin{equation*}
	D_l(v,t) = \big\{ \killedUpdateAbrv(v,t)\subset B_{l}(v), \tau_v(t)\leq \epsilon l+t \big\}.
\end{equation*}
Lemma~\ref{lem:ptToBoxDec} gives
\begin{equation}
\label{eq:BoxConstraintMem}
	\p\big( D_l(v,t) \big) \geq 1-\Cr{expDecUPD_cst}e^{-\Cr{expDecUPD_rate}l}.
\end{equation}

Define now an event controlling the ``cluster'' of a space time bloc:
\begin{equation*}
	A_l(w,t') = \bigcap_{v\in B_l(w)}\bigcap_{t\in[t',t'+\epsilon l]} \big\{ \killedUpdateAbrv(v,t)\subset B_{\frac{3}{2} l}(w), \tau_v(t)\leq t'+\frac{3}{2}\epsilon l \big\},
\end{equation*}
and denote \(A_l\equiv A_l(0,0)\). Notice that \(A_l(w,t')\) is measurable with respect to the sigma-algebra generated by the restriction of the Poisson point process to \(B_{\frac{3}{2}l}(w)\times [t',t'+\frac{3}{2}\epsilon l]\) and the associated \(U_i^v\)'s.

The next Lemma is the last control on the geometry of \(\killedUpdateAbrv\) that we will need to do our analysis.
\begin{lemma}
	\label{lem:GoodBoxUnif}
	There exist \(\Cl{cst_n2}\geq 0\) such that for any \(v\in\torusd\), any \( t\geq 0 \) and \(l\geq \Cr{cst_n2}\),
	\begin{equation*}
		\p(A_l(w,t')) \geq 1- e^{- \Cr{expDecUPD_rate}l/4}.
	\end{equation*}
\end{lemma}
\begin{proof}
	It is sufficient to prove the result for \(A_l\). Notice that \(A_l^c\) implies the existence of a point \((v,t)\in B_l\times [0, \epsilon l] \) such that
	\begin{itemize}
		\item \((v,t)\) is a point in the Poisson point process.
		\item \( D_{l/2}(v,t)^c\) occurs.
	\end{itemize}
	Discretize the time dimension of \(B_l\times [0,\epsilon l]\) by intervals of length \(\delta\) (we thus have \(|B_l|\frac{\epsilon l}{\delta}\) intervals). The lastly mentioned event implies the existence of at least one interval containing at least one point \((v,-t)\) of the PPP with \(D_{l/2}(v,t)^c\) occurring. If there is only one point of the PPP, \((v,-t)\), in a given interval, the event \(D_{l/2}(v,t)^c\) occurs if and only if \(D_{l/2}(v,\bar{t})^c\) occurs where \(\bar{t}= \lceil t/\delta\rceil \delta\). For a given interval \(I\), denote \(W_I^1\) the event that \(I\) contains exactly one point \((v,-t)\) and that \(D_{l/2}(v,\bar{t})^c\) occurs (notice that \(\bar{t}\) depends only on \(I\)) and denote \(W_I^2\) the event that \(I\) contains at least two points. This chain of upper bounds and a union bound give
	\begin{multline*}
		1-\p(A_l) \leq \sum_{I} \p(W_I^1) + \p(W_I^2)\leq\\
		\leq \sum_{I} \delta e^{-\delta}\p(D_{l/2}^c) + \frac{\delta^2}{2}
		= |B_l|\epsilon l e^{-\delta} \p(D_{l/2}^c) + \frac{\delta}{2} |B_l|\epsilon l,
	\end{multline*} for any \(\delta>0\). Letting \(\delta\to 0\) and using \eqref{eq:BoxConstraintMem}, we obtain
	\begin{equation*}
		1-\p(A_l) \leq (2l+1)^d\epsilon l \Cr{expDecUPD_cst}e^{-\Cr{expDecUPD_rate}l/2}.
	\end{equation*}
\end{proof}

We now partition the semi-continuous space \(\torusd_N\times(-\infty,0]\) in boxes as follows: take \(L>0\) even. We suppose \(N\) is divisible by \(2L+1\) and write \(N/(2L+1)=K\). Define the ``\((d+1)\)-dimensional half-lattice'':
\begin{equation*}
	\Gamma = \bbL_L\times \{k\epsilon L :\ k\in\Z_+ \}.
\end{equation*}
We see \(\bbL_L\) as a subset of \(\Gamma\) through the identification with \(\bbL_L\times \{0\}\). Define the boxes
\begin{equation*}
	B_L'(v,t) = B_L(v)\times [t,t+\epsilon L].
\end{equation*}
We equip \(\Gamma\) with the following graph structure: the vertices are the sites in \(\Gamma\) and two vertices \((v,t),(v',t')\) form an edge if
\begin{itemize}
	\item \(t=t'\) and \(\norm{v-v'}_{\infty}=2L+1\),
	\item \(\vert t-t'\vert=\epsilon L\) and \(\norm{v-v'}_{\infty}=2L+1\),
	\item \(\vert t-t'\vert=\epsilon L\) and \(v=v'\).
\end{itemize}
Notice that the set of boxes \(B'_L(v,t)\) with \((v,t)\in \Gamma\) forms a covering of \(\torusd_N\times [0,\infty)\) and that two boxes in this covering intersect only if they are nearest neighbours in the graph obtained from \(\Gamma\). We can identify (as graphs) \(\Gamma\) with \(\torusd_{K}\times \Z_+\). Connections in \(\Gamma\) are exactly \(*\)-connections in \(\torusd_{K}\times \Z_+\) (i.e.: put an edge between two sites if \(\norm{i-j}_\infty= 1\)). Denote \(d_{\Gamma}\) the graph distance in this graph.

Now, a site \((w,t)\in\Gamma\) is called \emph{good} if \(A_L(w,t)\) occurs and \emph{bad} otherwise. One obtains a site percolation configuration \(\omega\) on \(\Gamma\) by setting \(\omega_{z} =1-\mathds{1}_{A_L(z)}\) (the open sites correspond to bad boxes). Lemma~\ref{lem:GoodBoxUnif} together with the fact that the state of a site \((w,t)\) is independent of the states of the sites outside of \(B_{3L}(w)\times[t-\frac{3}{2}\epsilon L, t+\frac{3}{2}\epsilon L]\) implies that uniformly over the value of \((\omega_w)_{d_{\Gamma}(w,v)>1}\), the probability of \(\{\omega_v=0\}\) is greater or equal to \(1-e^{- \Cr{expDecUPD_rate}L/4}\). We then use~\cite[Theorem 1.3]{Liggett+Schonmann+Stacey-1997} to obtain that this site percolation is dominated by a Bernoulli percolation of parameter \(p = e^{-\Cl{rate_BernoulliPerco}L}\) where \(\Cr{rate_BernoulliPerco}>0\) depends only on \(\Cr{expDecUPD_rate}/4\) and on the maximal degree of \(\Gamma\) (\(3^{d+1}-1\)). Denote \(\eta\sim P_p\) a Bernoulli percolation of parameter \(p\) on the sites of \(\Gamma\).

For any \(z\in\Gamma\), define \(C_z(\eta) = \{v\in\Gamma:\ z\xleftrightarrow{\eta} v  \}\) the cluster of \(z\) (in \(\eta\), it is empty if \(\eta_z=0\)) and \(\partial C_{z}\) the set of sites in \(\Gamma\setminus C_z\) neighbouring a site of \(C_z\). By convention, set \(\partial C_z =\{z\}\) if \(C_z=\varnothing\). In the same spirit, for \(\Delta\subset \Gamma\) define \(C_{\Delta}=\bigcup_{z\in\Delta} C_z\). For a set of sites \(\Delta\subset \Gamma\), define its \emph{projection on space} by:
\begin{equation*}
	Proj(\Delta) = \big\{v\in\bbL_L :\ \exists t\geq 0, (v,t)\in \Delta \big\}.
\end{equation*}

The interest of all these objects is the inclusion (for any \(V\subset\bbL_L\), \(\bar{V}=\bigcup_{v\in V}B_L(v) \)):
\begin{equation*}
	\killedUpdateAbrv(\bar{V})\subset \bigcup_{w\in Proj(C_{V}\cup\partial C_V)} B_{3L/2}(w),
\end{equation*}
where we implicitly supposed that the process of good boxes and the Bernoulli percolation are sampled via an ordered coupling. In particular, \([\killedUpdateAbrv(\bar{V})]_L\) is included in the spatial projection of the set of sites at distance at most \(1\) (for \(d_{\Gamma}\)) of \(C_{V}\cup\partial C_V\). The cardinality of the latter spatial projection is then less than \(5^d\big(|Proj(C_V)|+|V|\big) \leq 5^{d}\big(|C_V|+|V|\big)\). One has thus the bound
\begin{equation}
\label{eq:BoxesToBernPerco}
	\p(|[\killedUpdateAbrv(\bar{V})]_L| \geq M)\leq P_{p}\big(|C_V| \geq (M-|V|)/5^d\big).
\end{equation}

\begin{lemma}
	\label{lem:ExpDecVol_BernPerco}
	There exists \(L_0\geq 0\) depending on \(\Cr{rate_BernoulliPerco},d\) only such that for \(L\geq L_0\), any set \(V\subset \bbL_L\times\{0\}\) and \(M\geq |V|\),
	\begin{equation*}
	P_{p}\big(|C_{V}|\geq M\big)\leq e^{-\frac{\Cr{rate_BernoulliPerco}}{2} M L}.
	\end{equation*}
\end{lemma}
\begin{proof}
	Fix a total order on \(V\), \(V=\{v_1,...,v_{|V|}\}\). The number of cluster of \(V\) with \(K\) vertices is smaller than or equal to the number of collections \((a_1,...,a_{|V|})\) of lattice animals with \(\sum_{i=1}^{|V|} |a_i|=K\) and \(v_i\in a_i\) or \(a_i=\varnothing\), for \(\ i=1,...,|V| \). To see this we construct an injection from the former set to the latter. The mapping goes as follows: for \(C_{V}=C\) fixed,
	\begin{enumerate}
		\item Set \(k=1\), \(I=\{1,...,|V|\}\).
		\item \label{it:step2}Take \(a_k\) to be the connected component of \(v_k\) in \(C\).
		\item Update \(I=I\setminus\{k\}\), \(C=C\setminus a_k\).
		\item Update \(k=k+1\) and go to step \ref{it:step2}.
	\end{enumerate}
	Then, using classical bounds on rooted lattice animals (see \cite[Section 4.2]{Grimmett-1999}), the number of collections \((a_1,...,a_{|V|})\) with the wanted properties is bounded from above by (for \(K\geq |V|\)),
	\begin{equation*}
		\sum_{\substack{k_1,...,k_{|V|}\geq 0 \\ \sum k_i = K}} \prod_{i=1}^{|V|} e^{c_d k_i}= e^{c_d K} \binom{K+|V|-1}{K}\leq \Big(\frac{e^{c_d+1}2K}{K}\Big)^K\equiv e^{cK}.
	\end{equation*}
	Going back to the initial claim,
	\begin{equation*}
		P_{p}\big(|C_{V}|\geq M\big) \leq \sum_{K=M}^{\infty} p^{K} e^{cK}
		= e^{-(\Cr{rate_BernoulliPerco}-c/L)L M}\sum_{K=0}^{\infty} e^{-(\Cr{rate_BernoulliPerco}L-c)K}.
	\end{equation*}
	Choosing \(L_0\) large enough implies the Lemma.
\end{proof}

In particular, we have that for any \(M\geq 0\), \(P_{p}\big(|C_{V}|\geq M\big)\leq e^{-\frac{\Cr{rate_BernoulliPerco}}{2} (M-|V|) L}\) and thus for any \(M\geq 0\),
\begin{equation*}
	\p(|[\killedUpdateAbrv(\bar{V})]_L| \geq M)\leq e^{-\frac{\Cr{rate_BernoulliPerco}}{2\cdot 5^d} LM}e^{ \frac{\Cr{rate_BernoulliPerco} (5^d +1)}{2 \cdot 5^d} L|V|}
\end{equation*}
by~\eqref{eq:BoxesToBernPerco} (where \(L\) has to be taken large enough to apply all Lemmas in this Section). This concludes the proof of Lemma~\ref{lem:ExpDecKilledUpdate}.

\section{Concluding the Proof of Theorems~\ref{thm:PressureAnalytic} and~\ref{thm:CorrFunctAnalytic}}
\label{sec:AnalyticityProof}
\subsection{Analyticity of the Pressure}

We start by showing how to use a well behaved dependency encoding measure to obtain analyticity of \(\pressure\) in \(\beta\) around a point \((\beta,h)\) such that \(\mu_{\beta,h}\) is weak mixing.
We start by a trivial equality:
\begin{equation}
\label{eq:trivialEquality}
\pressure(\beta+z,h) = \pressure(\beta,h) + \lim_{N\to\infty} \frac{1}{|\torusd_N|} \log\Big( \frac{Z_{N,\beta+z,h}}{Z_{N,\beta,h}}\Big).
\end{equation}
Then, \(\frac{Z_{N,\beta+z,h}}{Z_{N,\beta,h}} = \lrangle{e^{\sum_{\{i,j\}}z\sigma_i\sigma_j }}_{\beta,h} \equiv F_N(z)\). From now on, we see \(\bbL_L\) as a graph by adding edges between \(v\) and \(w\) when \(\normI{v-w}= 2L+1\). Denote this graph \((\bbL_L,E_{\bbL_L})\). One can write
\begin{align*}
	\sum_{\{i,j\}}z\sigma_i\sigma_j &= \sum_{v\in \bbL_{L} } \sum_{ \substack{\{i,j\}\subset B_L(v)\\i\sim j} } z\sigma_i\sigma_j + \sum_{ \substack{\{v,w\}\subset\bbL_L\\ v\sim w} } \sum_{\substack{i\in \partial^i B_L(v),\\ j\in \partial^i B_L(w),\\ i\sim j} } z\sigma_i\sigma_j,\\
	&\equiv \sum_{v\in \bbL_{L} } z f_v(\sigma) + \sum_{ \substack{\{v,w\}\subset\bbL_L\\ v\sim w} }zf_{vw}(\sigma).
\end{align*}
where \(\partial^{i} B_L = \{u\in B_L:\ \exists v\in \Zd\setminus B_L, u\sim v \} \). Now, we see \(\mu_{N}\) as a measure on \((\{-1,+1\}^{B_L})^{\bbL_L}\). Denote \(\Phi^L\) the measure constructed in Section~\ref{sec:InfoPerco}. With the way we chose to see \(\mu\), \(\Phi^L\) is dependency encoding for it.

Using~\eqref{eq:PolymerRep_General}, one obtains
\begin{equation*}
	F_N(z) = 1+\sum_{m\geq 1} \frac{1}{m!} \sum_{C_1,...,C_m} \prod_{\{i,j\}\subset[m]}\delta(C_i,C_j) \prod_{i=1}^m \weight(C_i)
\end{equation*}with \(C_1,...,C_m \subset \bbL_L\), \(\delta(C_i,C_j)=\mathds{1}_{C_i\cap C_j =\varnothing}\) and
\begin{equation*}
	\weight(C) = \sum_{H\subset( E_C\cup C)} \Phi\Big( \prod_{b\in H}(e^{z f_b}-1) \mathds{1}_{A( C,H )} \Big),
\end{equation*}where \(E_C\) is the set of edges with both endpoints in \(C\). Notice that \(\weight(C)=0\) if \(C\) is not connected. We will check the hypotheses of Theorem~\ref{thm:ClusterExp_Convergence} with \(g(C)=|C|\). First notice that for any \(b\in (E_{\bbL_L}\cup\bbL_L)\) and any \(\sigma\),
\begin{equation*}
	|f_b(\sigma)|\leq 2d(2L+1)^d\leq 6^d L^d.
\end{equation*}
Let \(1\geq \epsilon>0\) be small to be chosen later. Then, for \(|z|\leq \epsilon\), writing \(V_H = \bigcup_{b\in H} b \subset \bbL_L\) and \(\bar{V}_H=\bigcup_{v\in V_H} B_L(v)\subset \torusd_N\), for any \(a\in [0,1]\),
\begin{multline}
\label{eq:weightBound_1}
	|\weight(C)| \leq \sum_{H\subset (E_C\cup C)} (e^{6^d L^d}\epsilon)^{|H|}\Phi\big( |[\killedUpdate(\bar{V}_H)]_L| \geq |C|\big)\leq\\
	\leq \sum_{|H|\leq a|C|} e^{-\Cr{expDecKUPD_L} |C| L} e^{\Cr{SetVolCostKUPD_L}2|H|L} + (e^{6^d L^d}\epsilon)^{a |C|}\sum_{k= a|C|}^{(2d+1)|C|} \binom{(2d+1)|C|}{k}
\end{multline}
where we supposed \(\epsilon\leq e^{-6^dL^d}\) and used Lemma~\ref{lem:ExpDecKilledUpdate}. The first term is then less than \(e^{-\Cr{expDecKUPD_L} |C| L} e^{\Cr{SetVolCostKUPD_L}2a|C|L}2^{(2d+1)|C|}\). The sum in the second term is less than
\begin{equation*}
	\sum_{k= a|C|}^{3d|C|-1} \Big(\frac{e3d|C|}{k}\Big)^{k} \leq \sum_{k= 0}^{3d|C|-1} \Big(\frac{e3d}{a}\Big)^{k} \leq \Big(\frac{e3d}{a}\Big)^{3d|C|}.
\end{equation*}

We now check the hypotheses of Theorem~\ref{thm:ClusterExp_Convergence} for \(g(C)=|C|\). Remember that \(\weight(C)=0\) is \(C\) is not connected. Then, the number of connected \(C\) containing a given point with \(|C|=K\) is less than or equal to \(e^{c_dK}\) (see \cite[Section 4.2]{Grimmett-1999}) for some \(c_d\geq 0\) depending on \(d\). Now, for a fixed connected set of sites \(C'\),

\begin{equation*}
	\sum_{C \textnormal{ connected}} e^{|C|} |\weight(C)| \mathds{1}_{C\cap C'\neq \varnothing} \leq |C'|\sum_{ C\ni 0 \textnormal{ connected}} e^{|C|} |\weight(C)|
\end{equation*}
by translation invariance of \(\bbL_L\) and \(\weight\). The wanted inequality will be verified if we can show that \(|\weight(C)|\leq 2\Big(\frac{e^{-(c_d +1)}}{4}\Big)^{|C|} \equiv 2e^{-c' |C|}\). Choosing \(a=\frac{\Cr{expDecKUPD_L}}{4\Cr{SetVolCostKUPD_L}}\) and \(L=\max(L_0,\frac{4(2d+1)\log(2)}{\Cr{expDecKUPD_L}},\frac{4 c'}{\Cr{expDecKUPD_L}} )\) where \(L_0\) is given by Lemma~\ref{lem:ExpDecKilledUpdate}, we get that the first term in~\eqref{eq:weightBound_1} is less than \(e^{-c'|C|}\). Choosing then
\begin{equation*}
	\epsilon\leq \exp{-\frac{1}{a}\big( c' + a6^dL^d + 3d\log(e3d/a) \big)},
\end{equation*}
one gets that the second term in~\eqref{eq:weightBound_1} is less than \(e^{-c'|C|}\), implying the wanted bound. We thus have that \(\log(F_N(z))\) is analytic on \(\{|z|<\epsilon\}\) (see for example \cite[Theorem 5.8]{Friedli+Velenik-2017}). We thus have that \(\frac{1}{|\torusd_N|}\log(F_N(z))\) is a sequence of functions (indexed by \(N\)) that are all analytic on \(\{|z|<\epsilon\}\). Moreover, they form a family uniformly bounded on \(\{|z|<\epsilon\}\). By existence of \(\pressure(\beta+z,h)\) for \(z\in \R\), the sequence converge on a set having a cluster point in \(\{|z|<\epsilon\}\). Thus, Vitali Convergence Theorem implies that \(\lim_{N\to \infty} \frac{1}{|\torusd_N|}\log(F_N(z))\) is analytic. So, by~\eqref{eq:trivialEquality}, \(\beta'\mapsto\pressure(\beta',h)\) is analytic in a neighbourhood of \(\beta\).

The analyticity in \(h\) goes the same way (with some simplifications). The same procedure yields analyticity of the pressure in any small enough finite range perturbation of the potential.

\subsection{Analyticity of Multi-Point Functions}

We concentrate on the analyticity in \(h\), the analyticity in \(\beta\) follows in the same fashion. Fix \(A\subset \Zd\) finite. Then, for \(N\) large enough and \(L\) chosen in the same fashion as in the previous Section,
\begin{equation*}
	\lrangle{\sigma_A}_{N,\beta,h+z} = \frac{\lrangle{\prod_{v\in\bbL_L}e^{z M_{v}}f_v^A}_{N,\beta,h}}{\lrangle{\prod_{v\in\bbL_L}e^{z M_{B_L(v)}}}_{N,\beta,h}},
\end{equation*}
where \(M_v=\sum_{i\in B_L(v)}\sigma_i\) and \(f_v^A=\prod_{i\in B_L(v)\cap A}\sigma_i\). Define the weights
\begin{gather*}
	\weight(C) = \sum_{H \subset C} \Phi\Big( \prod_{v\in H}(e^{z M_v}-1) \mathds{1}_{A( C,H )} \Big)\\
	\weight^A(C) = \sum_{H \subset C} \Phi\Big( \prod_{v\in H}(e^{z M_v}f_v^A-1) \mathds{1}_{A( C,H )} \Big).
\end{gather*} We have \(\weight(C)=\weight^A(C)\) if \(A\cap C=\varnothing\). Using the cluster expansion as in the previous Section, one obtains:
\begin{multline*}
	\lrangle{\sigma_A}_{N,\beta,h+z} =\\ = \exp(\sum_{n\geq 1} \sum_{C_1}...\sum_{C_n} \Ursell(C_1,...,C_n) \big(\prod_{i=1}^n \weight^A(C_i)-\prod_{i=1}^n \weight(C_i)\big)),
\end{multline*}for all \(|z|\leq \epsilon\) with \(\epsilon\) such that we have convergence of the cluster expansion. The sum in the exponential being absolutely convergent, the LHS is analytic in \(z\) in a neighbourhood of \(0\) which is uniform over \(N\). Convergence of the sequence \(a_N=\lrangle{\sigma_A}_{N,\beta,h+z}\) for \(z\in \R\) implies the result as previously.

\begin{remark}
	One can alternatively use the fact that we have analyticity of the pressure in any local perturbation of the potential to obtain analyticity of the \(\lrangle{\sigma_A}_{\beta,h}\) as they are the derivative of the pressure of a locally perturbed model. The above way has the advantage of giving an ``explicit'' expression of \(\lrangle{\sigma_A}_{\beta+w,h+z}\).
\end{remark}

\section{Exponential Relaxation of the Magnetization}
\label{sec:ExpRelaxMagnet}
For this whole section, \(\beta\geq 0\) and \(h>0\) will be fixed and dropped from the notation. \(\mu_{N}^*\) and \(\lrangle{\ }_{N}^*\) will thus denote the law and expectation of the Ising model on \(\Lambda_N=[-N,N]^d\) with boundary condition \(*\). The goal of this section is the proof of
\begin{theorem}
	\label{thm:epxRelax_Magnet}
	There exist \(C\geq 0, \nu >0\) such that
	\begin{equation*}
		|\lrangle{\sigma_0}_{N}^{+} -\lrangle{\sigma_0}_{N}^{-}| \leq Ce^{-\nu N}.
	\end{equation*}
\end{theorem}

The first step is the proof of
\begin{lemma}
	\label{lem:epxRelax_MagnetPlusToFree}
	There exist \(C\geq 0, \nu >0\) such that
	\begin{equation*}
		|\lrangle{\sigma_0}_{N}^{+} -\lrangle{\sigma_0}_{N}^{0}| \leq Ce^{-\nu N}.
	\end{equation*}
\end{lemma}
\begin{proof}
	Denote \(\lrangle{\sigma_0}_{N}^{s}\) the law of the Ising model with \(+\) boundary conditions and coupling constants \(s\) on the edges \(\big\{ \{i,j\}:\ i\in \Lambda_N, j\in\Zd\setminus\Lambda_N, i\sim j \big\} \). We thus have
	\begin{multline*}
		0\leq \lrangle{\sigma_0}_{N}^{+} -\lrangle{\sigma_0}_{h}^{0} = \int_{0}^{1} \frac{d}{ds}\lrangle{\sigma_0}_{N}^{s} ds\\
		= \beta\sum_{x\in \partial^{\interior} \Lambda_N }\int_{0}^{1} \lrangle{\sigma_0;\sigma_x}_{N}^{s} ds
		\leq \beta\sum_{x\in \partial^{\interior} \Lambda_N } e^{-cN},
	\end{multline*}
	where the inequality in the first line is FKG and the last inequality is the exponential decay property of the Ising model with a field (see \cite{Lebowitz+Penrose-1968}). \(\partial^{\interior} \Lambda_N\) denotes the set of sites in \(\Lambda_N\) sharing an edge with a site in \(\Zd\setminus \Lambda_N\).
\end{proof}

The next step is the more complicated one. The proof relies on the random cluster representation of the Ising model with a field that we now present. Let \(G=(V,E)\) be a finite graph. Consider the measure on the subsets of \(E\) given by:
\begin{equation*}
	\p_{G}(\omega) \propto (e^{2\beta} -1)^{\omega} \prod_{C\in \clusterSet(\omega)} (e^{h|C|}+e^{-h|C|}),
\end{equation*}
where \(\clusterSet(\omega)\) is the set of connected components of \(\omega\) and \(|C|\) is the number of sites in \(C\). This measure has the properties:
\begin{itemize}
	\item FKG inequality for the canonical order on the subsets of \(E\): for any \(f,g\) both non-decreasing functions,
	\begin{equation*}
		\e_G[fg] \geq \e_G[f]\e_G[g].
	\end{equation*}
	\item Finite energy:
	\begin{equation*}
		\frac{e^{2\beta}-1}{e^{2\beta}+1}\leq \sup_{e\in E} \sup_{\eta\subset (E\setminus \{e\})}\p_G(\omega_e=1\given \omega_f=\eta_f\ \forall f\neq e) \leq 1-e^{-2\beta}.
	\end{equation*}
	\item Closed edges are decoupling: for \(F\subset E\) a cut-set separating \(G\) in \(G_1\) and \(G_2\), \(f\) supported on \(G_1\) and \(g\) supported on \(G_2\),
	\begin{equation*}
		\e_G[fg\given \omega_e=0\ \forall e\in F ] = \e_{G_1}[f] \e_{G_2}[g].
	\end{equation*}
\end{itemize}
The link with the Ising model is the following one: sample \(\omega\sim \p_G\) and independently colour each cluster \(C\) of \(\omega\) \(+1\) with probability \(\frac{e^{h|C|}}{e^{h|C|}+e^{-h|C|}}\) and \(-1\) with probability \(\frac{e^{-h|C|}}{e^{h|C|}+e^{-h|C|}}\). Denote \(\sigma\in\{-1,+1\}^{V}\) the obtained configuration. Then, \(\sigma\sim \mu_{G,\beta,h}\). Denote \(P_G\) the law and expectation of the pair \((\omega,\sigma)\).
To take into account the boundary conditions, we use a modified graph. Take \(\Lambda\) a finite subgraph of \(\Zd\) and define \(\partial \Lambda\) the set of sites in \(\Zd\setminus \Lambda\) sharing an edge with a site in \(\Lambda\). Define then \(\Lambda_{\partial}\) to be the graph with vertex set \(\Lambda\cup \{\partial\}\) and edge set \(E\cup \bigcup_{j\in \partial \Lambda}\big\{ \{\partial, i\}:\ i\sim j\big\}\equiv E\cup E_{\partial}\). We obtain the Ising measure with \(-\) boundary condition via:
\begin{equation*}
	\lrangle{f}^{-}_{\Lambda} = P_{\Lambda_{\partial}}( f(\sigma) \given \sigma_{\partial} = -1).
\end{equation*}
In particular, if \(\Lambda=\Lambda_N\),
\begin{multline}
	\label{eq:magnet_FK_minusBC}
	\lrangle{\sigma_0}^{-}_{\Lambda} = P_{\Lambda_{\partial}}( 0\stackrel{\omega}{\nlongleftrightarrow} \partial,\sigma_0=+1 \given \sigma_{\partial} = -1)-\\-P_{\Lambda_{\partial}}( 0\stackrel{\omega}{\nlongleftrightarrow} \partial,\sigma_0=-1 \given \sigma_{\partial} = -1)-P_{\Lambda_{\partial}}( 0\xleftrightarrow{\omega} \partial \given \sigma_{\partial} = -1).
\end{multline}

Using this, we can now prove
\begin{lemma}
	\label{lem:epxRelax_MagnetMinusToFree}
	There exist \(C\geq 0, \nu >0\) such that
	\begin{equation*}
		|\lrangle{\sigma_0}_{N}^{0} -\lrangle{\sigma_0}_{N}^{-}| \leq Ce^{-\nu N}.
	\end{equation*}
\end{lemma}
\begin{proof}
	Let \(\Lambda=\Lambda_{N}\). First, by monotonicity of the Ising measure, \(\lrangle{\sigma_0}_{N}^{0} -\lrangle{\sigma_0}_{N}^{-}\geq 0\). Then, by~\eqref{eq:magnet_FK_minusBC},
	\begin{multline*}
		\lrangle{\sigma_0}_{N}^{-} = \e_{\Lambda_{\partial}}\big[ (e^{2h|C_{\partial}|}+1)^{-1}\big]^{-1} \Big(-\e_{\Lambda_{\partial}}\big[ \mathds{1}_{0\xleftrightarrow{\omega}\partial} (e^{2h|C_{\partial}|}+1)^{-1} \big] +\\+ \e_{\Lambda_{\partial}}\Big[\mathds{1}_{0\stackrel{\omega}{\nlongleftrightarrow} \partial} \frac{e^{h|C_0|}- e^{-h|C_0|}}{e^{h|C_0|}+e^{-h|C_0|}} (e^{2h|C_{\partial}|}+1)^{-1}\Big] \Big).
	\end{multline*}
	Defining \(p(\omega) = (e^{2h|C_{\partial}(\omega)|}+1)^{-1}\), we have
	\begin{equation*}
		\e_{\Lambda_{\partial}}\big[ p\big]\lrangle{\sigma_0}_{N}^{-} =  \e_{\Lambda_{\partial}}\Big[\mathds{1}_{0\stackrel{\omega}{\nlongleftrightarrow} \partial} \tanh(h|C_0|) p\Big] -\e_{\Lambda_{\partial}}\big[ \mathds{1}_{0\xleftrightarrow{\omega}\partial} p \big].
	\end{equation*}
	Now, for any \(K\) (dividing \(N\)),
	\begin{equation*}
		\e_{\Lambda_{\partial}}\Big[\mathds{1}_{0\stackrel{\omega}{\nlongleftrightarrow} \partial} \tanh(h|C_0|) p\Big] \geq \e_{\Lambda_{\partial}}\Big[\mathds{1}_{\partial\stackrel{\omega}{\nlongleftrightarrow} \Lambda_{N/K}} \tanh(h|C_0|) p\Big].
	\end{equation*}
	Then, \(\partial\stackrel{\omega}{\nlongleftrightarrow} \Lambda_{N/K}\) implies that the edge boundary of \(C_{\partial}\) (that we will denote \(\partial^{\edge}C_{\partial}\)) is included in \(\Lambda_{\partial}\setminus\Lambda_{N/K}\). Moreover, it is a cut set. The decoupling property of closed edges implies:
	\begin{align*}
		&\e_{\Lambda_{\partial}}\Big[\mathds{1}_{\partial\stackrel{\omega}{\nlongleftrightarrow} \Lambda_{N/K}} \tanh(h|C_0|) p\Big] =\\
		&\quad = \sum_{\substack{C\ni \partial\\ C\cap \Lambda_{N/K}=\varnothing}} (e^{2h|C|}+1)^{-1} \p_{\Lambda_{\partial}}\big(C_{\partial} = C\big) \e_{\Lambda_{\partial}\setminus C}\big[\tanh(h|C_0|)\big]\\
		&\quad = \sum_{\substack{C\ni \partial\\ C\cap \Lambda_{N/K}=\varnothing}} (e^{2h|C|}+1)^{-1} \p_{\Lambda_{\partial}}\big(C_{\partial} = C\big) \lrangle{\sigma_0}_{\Lambda_{\partial}\setminus C}^0.
	\end{align*}
	When writing \(\Lambda_{\partial}\setminus C\) we mean the graph obtained from \(\Lambda_{\partial}\) by removing the vertices of \(C\) (and the associated edges). Now, the monotonicity of the Ising model (in the volume, it is a direct consequence of GKS inequality) and the constraint on \(C\) implies \(\lrangle{\sigma_0}_{\Lambda_{\partial}\setminus C}^0 \geq \lrangle{\sigma_0}_{N/K}^0\). Thus,
	\begin{equation*}
		\e_{\Lambda_{\partial}}\Big[\mathds{1}_{\partial\stackrel{\omega}{\nlongleftrightarrow} \Lambda_{N/K}} \tanh(h|C_0|) p\Big] \geq \lrangle{\sigma_0}_{N/K}^0 \e_{\Lambda_{\partial}}\big[ \mathds{1}_{\partial\stackrel{\omega}{\nlongleftrightarrow} \Lambda_{N/K}} p\big].
	\end{equation*}
	We obtain the upper bound (as \(\{0\xleftrightarrow{\omega} \partial\}\subset \{\partial \xleftrightarrow{\omega} \Lambda_{N/K}\}\) and \(\sigma_0\leq 1\)):
	\begin{equation*}
		-\lrangle{\sigma_0}_{N}^{-} \leq - \lrangle{\sigma_0}_{N/K}^0 + 2 P_{\Lambda_{\partial}}\big( \partial\xleftrightarrow{\omega} \Lambda_{N/K}\given \sigma_{\partial} = -1\big).
	\end{equation*}
	We have,
	\begin{equation*}
		\lrangle{\sigma_0}_{N}^{0} -\lrangle{\sigma_0}_{N}^{-} \leq \lrangle{\sigma_0}_{N}^{0} - \lrangle{\sigma_0}_{N/K}^0 + 2 P_{\Lambda_{\partial}}\big( \partial\xleftrightarrow{\omega} \Lambda_{N/K}\given \sigma_{\partial} = -1\big).
	\end{equation*}
	As,
	\begin{equation*}
		0\leq \lrangle{\sigma_0}_{N}^{0} - \lrangle{\sigma_0}_{N/K}^0 \leq \lrangle{\sigma_0}_{N/K}^{+} - \lrangle{\sigma_0}_{N/K}^0 \leq Ce^{-\frac{\nu}{K} N },
	\end{equation*}
	by FKG and Lemma~\ref{lem:epxRelax_MagnetPlusToFree}, Lemma~\ref{lem:epxRelax_MagnetMinusToFree} will follow from
	\begin{claim}
		There exist \(K\geq 0, \nu>0, N_0\geq 0\) such that for any \(N\geq N_0\) divisible by \(K\),
		\begin{equation*}
			P_{\Lambda_{\partial}}\big( \partial\xleftrightarrow{\omega} \Lambda_{N/K}\given \sigma_{\partial} = -1\big)\leq e^{-\nu N}.
		\end{equation*}
	\end{claim}
	\begin{proof}
		We start by a simple observation. As \(\frac{1}{2}\leq (e^{-2h k}+1)^{-1}\leq 1\), for any event \(A\),
		\begin{equation*}
			\frac{1}{2} \frac{\e_{\Lambda_{\partial}}[e^{-2h|C_{\partial}|}\mathds{1}_{\omega\in A}]}{\e_{\Lambda_{\partial}}[e^{-2h|C_{\partial}|}]} \leq \frac{\e_{\Lambda_{\partial}}[p\mathds{1}_{\omega\in A}]}{\e_{\Lambda_{\partial}}[p]} \leq 2\frac{\e_{\Lambda_{\partial}}[e^{-2h|C_{\partial}|}\mathds{1}_{\omega\in A}]}{\e_{\Lambda_{\partial}}[e^{-2h|C_{\partial}|}]}.
		\end{equation*}
		We can prove the result by proving
		\begin{equation*}
			\bfP_{\Lambda_{\partial}}(\partial\xleftrightarrow{\omega} \Lambda_{N/K}\given C_{\partial} \textnormal{ free})\leq Ce^{-\nu N},
		\end{equation*}where \(\bfP\) is the product law of \(\p_{\Lambda_{\partial}}\) and a i.i.d. family of Bernoulli random variables of parameter \(1-e^{-2h}\) indexed by the vertices of \(\Lambda_{\partial}\) and \(\{C_{\partial} \textnormal{ free}\}\) is the event that all Bernoulli random variables attached to a site in \(C_{\partial}(\omega)\) are \(0\).
		
		For \(i=1,...,d\), define \(A_{i}\) the event that there exist at most \( N^{d-i}\) disjoints open paths belonging to \(C_{\partial}\) and going from \(\partial^{\interior}\Lambda_{N/2^{i-1}}\) to \(\Lambda_{N/2^{i}}\). Notice that when \(A_i\) occurs, there exists a cut-set of edges \(\gamma\) separating \(\partial\) from \(\Lambda_{N/2^{i}}\) containing less than \(N^{d-i}\) open edges. On \(A_i\), we can thus define \(\Gamma_i\) to be the (random) set of edges obtained as follows: take \(E_i\) to be the outermost (closest to \(\partial\)) cut-set of \(C_{\partial}\) such that: \(|E_i|\leq N^{d-i}\) and \(E_i\subset \Lambda_{N/2^{i-1}}\setminus \Lambda_{N/2^{i}}\). Let \(C_i\subset C_{\partial}\) be the component of the cut defined by \(E_i\) containing \(\partial\). Set \(\Gamma_i=\partial^{\edge}C_i\). Denote also \(\Delta_i=\Delta(\Gamma_i)\) the graph obtained from \(\Lambda_{\partial}\) by removing \(\Gamma_i\) and taking the connected component of \(\Lambda_{N/2^{i}}\) and \(V_i\) the set of vertices in \(\Delta_i\) that are endpoint of an edge in \(E_i\). We have the identity:
		\begin{equation*}
			e^{-2h|C_{\partial}(\omega)|} = e^{-2h|C_i(\omega)|}e^{-2h|C_{V_i}(\omega_{\Delta_i})|},
		\end{equation*}
		where \(C_{V_i}\) is the cluster of \(V_i\) and \(\omega_{\Delta_i}\) is the restriction of \(\omega\) to \(\Delta_i\).
		
		Then, define the mapping \(T_i:A_i\to \{\partial^{\interior}\Lambda_{N/2^{i-1}}\nleftrightarrow \Lambda_{N/2^i}\}\) that closes all edges in \(\Gamma_i\).
		
		We now want to prove that there exist \(\nu_i>0, i=1,...,d\) such that (denote \(A_0\) the full space of configurations)
		\begin{equation}
			\label{eq:Ai_expBnd}
			\bfP_{\Lambda_{\partial}}(A_i^c,A_{i-1}, C_{\partial} \textnormal{ free}) \leq e^{-\nu_i N}\bfP_{\Lambda_{\partial}}(A_{i-1}, C_{\partial} \textnormal{ free}).
		\end{equation}
		We start with the \(i=1\) case. One has the a-priori bound:
		\begin{equation*}
			\bfP_{\Lambda_{\partial}}(C_{\partial} \textnormal{ free})\geq e^{-2h}e^{-2\beta 2d(2N+1)^{d-1}},
		\end{equation*}by finite energy (we can close all edges between \(\partial\) and \(\Lambda_N\)). On the other hand, \(A_1^c\) implies that \(|C_{\partial}|\geq N^{d-1}\frac{N}{2}\). Thus, \(\bfP_{\Lambda_{\partial}}(A_1^c, C_{\partial} \textnormal{ free})\leq e^{-h N^{d}}\). This implies that there exists \(\nu_1>0\) such that
		\begin{equation*}
			\bfP_{\Lambda_{\partial}}(A_1^c \given C_{\partial} \textnormal{ free})\leq e^{-\nu_1 N^{d}},
		\end{equation*}whenever \(N\) is large enough. For \(1<i\leq d\), we use the mapping \(T_{i-1}\) to implement a many to one argument.
		\begin{multline*}
			\bfP_{\Lambda_{\partial}}(A_i^c,A_{i-1}, C_{\partial} \textnormal{ free}) \leq\\
			\leq e^{c N^{d-i} } \binom{2d(2N+1)^d}{ N^{d-i}}\bfP_{\Lambda_{\partial}}(A_{i-1}, C_{\partial} \textnormal{ free}) e^{-2h 2^{-i} N^{d-i+1}}
		\end{multline*}
		where we used the many to one principle, finite energy, that to reconstruct \(\omega\) from \(T_{i-1}(\omega)\) one need to specify which edges where closed by \(T_{i-1}\), there are at most \(N^{d-i}\) such edges and they have to belong to \(\Lambda_N\) and that \(A_i^c\) implies that \(|C_{V_{i-1}}(\omega_{\Delta_{i-1}})|\geq  2^{-i}N^{d-i+1}\). Now,
		\begin{equation*}
			\binom{2d(2N+1)^d}{ N^{d-i}} \leq \big(2d e 3^d N^i\big)^{ N^{d-i}} = e^{c_i \log(N) N^{d-i}}.
		\end{equation*} This implies the existence of \(\nu_i>0\) such that, for any \(N\) large enough,
		\begin{equation*}
			\bfP_{\Lambda_{\partial}}(A_i^c,A_{i-1}, C_{\partial} \textnormal{ free}) \leq e^{-\nu_i N^{d-i+1}}\bfP_{\Lambda_{\partial}}(A_{i-1}, C_{\partial} \textnormal{ free}).
		\end{equation*}
		Taking \(N\) large enough so that~\eqref{eq:Ai_expBnd} is satisfied for \(i=1,...,d\), one has:
		\begin{equation*}
			\bfP_{\Lambda_{\partial}}(A_i^c,A_{i-1}\given C_{\partial} \textnormal{ free}) \leq e^{-\nu_i N}.
		\end{equation*}		
		In particular,
		\begin{equation*}
			\bfP_{\Lambda_{\partial}}(A_d^c\given C_{\partial} \textnormal{ free}) \leq \sum_{i=1}^d e^{-\nu_i N}.
		\end{equation*}
		Setting \(K=2^{d+1}\) and doing the same one to many argument as before using \(T_d\) and the fact that a crossing from \(\partial^{\interior}\Lambda_{N/2^d}\) to \(\Lambda_{N/K}\) uses at least \(N/K\) sites, one gets
		\begin{equation*}
			\bfP_{\Lambda_{\partial}}(\partial\xleftrightarrow{\omega} \Lambda_{N/K}, A_d \given C_{\partial} \textnormal{ free})\leq e^{-\nu N},
		\end{equation*} for some \(\nu>0\) whenever \(N\) is large enough. This concludes the proof.
	\end{proof}
\end{proof}

\section{Acknowledgments}
The author thanks Yvan Velenik for various comments on a previous draft of this paper. The author gratefully acknowledge the support of the Swiss National Science Foundation through the NCCR SwissMAP.

\appendix

\section{Cluster Expansion}
\label{app:ClusterExp}

We recall here what is the cluster expansion of a pair interaction polymer model and a result about convergence of this expansion. The whole presentation can be found in \cite{Friedli+Velenik-2017} so we only state the results and refer to \cite[Chapter 5]{Friedli+Velenik-2017} for proofs and more details.

\subsection*{The Framework}
Suppose we are given a set \(\Gamma\) (the set of polymers), a weighting \(\weight:\Gamma\to\mathbb{C}\) and an interaction \(\delta:\Gamma\times\Gamma\to[-1,1] \). The \emph{polymer partition function} is then given by
\begin{equation*}
	Z = \sum_{H \subset \Gamma \text{ finite}} \Big(\prod_{\gamma\in H}\weight(\gamma)\Big) \Big(\prod_{\{\gamma,\gamma'\}\subset H} \delta(\gamma,\gamma')\Big).
\end{equation*}
The empty set contributes \(1\) to the sum. To state the formal equality, we need to define the \emph{Ursell function} of an ordered collection of polymers:
\begin{equation*}
	\Ursell(\gamma_1) = 1,\quad \Ursell(\gamma_1,...,\gamma_n) = \frac{1}{n!} \sum_{\substack{G\subset K_n\\ \textnormal{connected}}} \prod_{\{i,j\}\in E_G} (\delta(\gamma_i,\gamma_j)-1),
\end{equation*}where \(K_{n}\) is the complete graph on \(\{1,...,n\}\), \(G=(\{1,...,n\},E_G)\) is an edge-subgraph of \(K_n\).
\subsection*{The Formal Equality} Equipped with this set-up, we have the equality (valid when the sum in the exponential is absolutely convergent)
\begin{equation}
\label{eq:ClusterExpansion}
	Z = \exp(\sum_{n\geq 1} \sum_{\gamma_1}...\sum_{\gamma_n} \Ursell(\gamma_1,...,\gamma_n) \prod_{i=1}^n \weight(\gamma_i)).
\end{equation}
\subsection*{Convergence}
The result we will use is the following criterion for the absolute convergence of \(\sum_{n\geq 1} \sum_{\gamma_1}...\sum_{\gamma_n} \Ursell(\gamma_1,...,\gamma_n) \prod_{i=1}^n \weight(\gamma_i)\):
\begin{theorem}
\label{thm:ClusterExp_Convergence}
	If there exists \(g:\Gamma\to\R_{>0}\) such that for every \(\gamma'\in\Gamma\)
	\begin{equation*}
		\sum_{\gamma\in\Gamma} e^{g(\gamma)} |\weight(\gamma)| |\delta(\gamma,\gamma')-1|\leq g(\gamma'),
	\end{equation*}
	and such that \(\sum_{\gamma\in\Gamma} e^{g(\gamma)}|\weight(\gamma)|<\infty\) then,
	\begin{equation*}
		\sum_{n\geq 1} \sum_{\gamma_1}...\sum_{\gamma_n} |\Ursell(\gamma_1,...,\gamma_n)| \prod_{i=1}^n |\weight(\gamma_i)|<\infty.
	\end{equation*}
\end{theorem}

\bibliographystyle{plain}
\bibliography{PressureAnalytic}

\end{document}